\documentclass[reqno]{amsart}
\usepackage{pgfplots}
\usepackage{tikz}
\usetikzlibrary{datavisualization}
\usetikzlibrary{graphs}
\usepackage{epsfig}
\usepackage{graphicx}
\usetikzlibrary{decorations.markings}
\usepackage{amsmath}
\usepackage{amssymb,latexsym}
\usepackage{amsthm}
\usepackage{mathrsfs}
\usepackage{subfig}
\usepackage[colorlinks,citecolor=blue]{hyperref}
\usepackage{caption}
\usepackage{color}
\usepackage{ulem}

\newtheorem{definition}{Definition}
\newtheorem{theorem}{Theorem}[section]
\newtheorem{lemma}{Lemma}[section]

\begin{document}
\baselineskip 15pt

\title[Orbital stability of smooth solitary waves]
{Orbital stability of smooth solitary waves for the $b$-family of Camassa-Holm equations}

\author[T. Long, C. Liu,]{Teng Long, Changjian Liu$^*$}

\address{Teng Long\newline
School of Mathematics (Zhuhai), Sun Yat-sen University. Zhuhai 519082, China}
\email{longt28@mail2.sysu.edu.cn}

\address{Changjian Liu\newline
School of Mathematics (Zhuhai), Sun Yat-sen University. Zhuhai 519082, China}
\email{liuchangj@mail.sysu.edu.cn}

\date{}


\begin{abstract}
In this paper, we study the stability of smooth solitary waves for the $b$-family of Camassa-Holm equations. We verify the stability criterion analytically for the general case $b>1$ by the idea of the monotonicity of the period function for planar Hamiltonian systems and show that the smooth solitary waves are orbitally stable, which gives a positive answer to the open problem proposed by Lafortune and Pelinovsky [S. Lafortune, D. E. Pelinovsky, Stability of smooth solitary waves in the $b$-Camassa-Holm equation].
\end{abstract}
\maketitle

\section{Introduction}

The $b$-family of Camassa-Holm equations (namely $b$-CH equation)
\begin{equation}\label{CH}
u_t-u_{txx}+(b+1)uu_x=bu_xu_{xx}+uu_{xxx},
\end{equation}
was mentioned by Degasperis and Dullin et al. in \cite{Degasperis2,Dullin} by using transformations of the integrable hierarchy of KdV equations, where $u=u(t,x)$ is the scalar velocity variable and $b$ is arbitrary parameter. The Camassa-Holm equation is used to describe the unidirectional propagation of water waves on a free surface in shallow water.

Degasperis and Procesi \cite{Degasperis1} proved that the $b$-CH equation \eqref{CH} is not integrable in general, but includes both the integrable cases of $b=2$ called the Camassa-Holm equation \cite{Camassa,Liu2} and $b=3$ called the Degasperis-Procesi equation \cite{Degasperis1,Lundmark} as special cases. Furthermore, Camassa-Holm equation and Degasperis-Procesi equation for important modelling of shallow water waves with breaking phenomena were studied in \cite{Constantin1,Constantin2,Liu,Whitham}.

%It is of great theoretical and practical significance to study the traveling waves, because it can well describe various physical phenomena such as vibration and propagating waves.
Travelling waves are usually divided into the following four categories: (i) stationary wave solution, (ii) travelling wavefront, (iii) soliton, (iv) periodic wave solution.
Moreover, peaked and smooth solitary waves exist in the $b$-CH equation \eqref{CH} and depend on parameter $b$ and parameter $k$ ($k$ is related to the critical shallow water speed), and both belong to solitary waves of (iii) soliton.

Many methods are available for solving the travelling waves. Using the qualitative theory of differential equations and the bifurcation method of dynamical systems to study traveling waves was first put forward by Liu and Li in \cite{Liu1} and they have shown that the solitary waves correspond to homoclinic orbits in the bifurcation phase diagrams of planar Hamiltonian systems.

In \cite{Guo}, the bifurcation method of dynamical systems and the numerical simulation approach of differential
equations are used to investigate traveling waves of the $b$-CH equation \eqref{CH}. Lately, the travelling waves of the $b$-CH equation \eqref{CH} was learned by Barnes and Hone in \cite{Barnes} by using hodograph transformation.

Regarding the stability theory of solitary waves, Grillakis, Shatah and Strauss showed an abstract and complete solitary wave orbit stability theory and found sharp conditions for the stability and instability of solitary waves in \cite{Grillakis}. Many results have been obtained for orbital stability of solitary waves for the $b$-CH equation \eqref{CH} with this method.

For zero asymptotic value ($k=0$), the initial data were decomposed into a sequence of peaked solitary waves called \textit{peakons} for $b>1$ and a sequence of smooth solitary waves called \textit{leftons} for $b<-1$ by numerical simulations in \cite{Holm1,Holm2}. For $b\in(-1,1)$, a rarefactive wave with exponentially decaying tails is generated from the initial data. The orbital stability of \textit{leftons} for $b<-1$ in some exponentially weighted space, \textit{peakons} for $b=2$ in the energy space $H^1(\mathbb{R})$ and $b=3$ in the energy space $L^2(\mathbb{R})\cap L^3(\mathbb{R})$ is studied in \cite{Hone}, \cite{Constantin3,Constantin4} and \cite{Lin}, respectively. %Recently, Lafortune and Pelinovsky \cite{Lafortune1} showed spectral and linear instability of peakons of the b-family of Camassa-Holm equations for every $b$ in $L^2(\mathbb{R})$.

For nonzero asymptotic value ($k\neq0$), Constantin, Strauss \cite{Constantin5} and Li, Liu, Wu \cite{Li} proved that the smooth solitary waves for $b=2$ and $b=3$ are orbitally stable by using the conserved energy integrals in the energy space, respectively. In addition, Liu et. al in \cite{Liu2} showed that for $b=2$ the Camassa-Holm equation has a peakon solution and orbital stability of the peakon solution was discussed in \cite{Ouyang}. Recently, Lafortune and Pelinovsky \cite{Lafortune} deduced a precise condition for orbital stability of the smooth solitory waves for the $b$-CH equation \eqref{CH} and verified the stability criterion analytically for $b=2$ and $b=3$ and numerically for every $b>1$. They said that it is still open to verify the stability criterion analytically for every $b>1$, $c>0$, and $k\in(0, \frac{c}{b+1})$, where $c$ is a constant wave speed.

Motivated by Lafortune and Pelinovsky \cite{Lafortune}, the main purpose of this article is to consider orbital stability of the smooth solitary waves of the $b$-CH equation \eqref{CH} for every $b>1$ by means of the monotonicity of the period function for planar Hamiltonian systems.

The main result of this paper is as follows:
\begin{theorem}
For every $b>1$, $c>0$ and $k\in(0, \frac{c}{b+1})$, the smooth solitary waves of the $b$-CH equation \eqref{CH} are orbitally stable.
\end{theorem}

The paper is organized as follows: preparations and the process of transformation of the $b$-CH equation \eqref{CH} to verify the stability criterion analytically for any $b>1$ are put in Section 2. In the Section 3, we present a consequence to deal with the stability criterion and Section 4 is devoted to the proof of the main result.

\section{Stability criterion and transformation of the b-CH equation}\label{sect-2}
In this section, we need the following preparations and lemmas that will be used throughout this paper, and we will show the transformation of the smooth solitary waves (the homoclinic orbits) into the first integral of planar Hamiltonian system.

If we plug $u(t,x)=\phi(x-ct)$ back into \eqref{CH}, then the $b$-CH equation \eqref{CH} becomes the following third-order differntial equation
\begin{equation}\label{ode3}
-(c-\phi)(\phi'''-\phi')+b\phi'(\phi''-\phi)=0.
\end{equation}
where $\phi:=\phi(x)$.
Integrating in $x$ yield the second-order equation:
\begin{equation}\label{ode4}
(c-\phi)(\phi-\phi'')+\frac{1}{2}(b-1)(\phi'^{2}-\phi^2)=ck-\frac{1}{2}(b+1)k^2.
\end{equation}

The following lemma summarizes the existence of smooth solitary waves for the $b$-CH equation \eqref{CH}.

\begin{lemma}\label{le1}{\rm(see \cite{Lafortune})}
For fixed $b>1$ and $c>0$, there exists a one-parameter family of smooth
solitary waves with profile $\phi\in C^{\infty}(\mathbb{R})$ satisfying $\phi'(0)=0$ and $\phi(x)\to k$ as $|x|\to\infty $ if
and only if the arbitrary parameter $k$ belongs to the interval $(0,\frac{c}{b+1})$. Moreover,
\begin{equation}\label{1}
0<\phi(x)<c, \ \ x\in\mathbb{R},
\end{equation}
and the family is smooth with respect to parameter $k$ in $(0,\frac{c}{b+1})$.
\end{lemma}

Let us now give the definition of orbital stability of the smooth solitary waves for the $b$-CH equation \eqref{CH}.

According to the analysis of \cite{Lafortune}, the $b$-CH equation \eqref{CH} takes the form by using the momentum density $m:=u-u_{xx}$
\begin{equation}\label{m}
m_t+um_x+bmu_x=0,
\end{equation}
where $m\in X_k$, $X_k=\{m-k\in H^1(\mathbb{R}): \ m(x)>0, x\in\mathbb{R}\}$, and $H^1(\mathbb{R})$ is based on the Sobolev space on $L^2(\mathbb{R})$.

\begin{definition}\label{de1}
Let $m(t,x)=\mu(x-ct)$ be the travelling wave solution of the b-CH equation \eqref{m} with $\mu\in X_{k}$. We say that the travelling wave is orbitally stable in $X_{k}$ if for every $\varepsilon>0$ there exists $\delta>0$ such that for every $m_0\in X_{k}$ satisfying $\|m_0-\mu\|_{H^1}<\delta$, there exists a unique solution $m\in C^0(\mathbb{R},X_k)$ of the $b$-CH equation \eqref{m} with the initial datum $m(0,\cdot)=m_0$ satisfying
$$\inf\limits_{x_0\in\mathbb{R}}\|m(t,\cdot)-\mu(\cdot-x_0)\|_{H^1}<\varepsilon, \ t\in\mathbb{R}.$$
\end{definition}

The following lemma gives the stability criterion of smooth solitary waves for the $b$-CH equations \eqref{CH}.

\begin{lemma}\label{le2}{\rm(see \cite{Lafortune})}
For fixed $b>1$, $c>0$, and $k\in(0,\frac{c}{b+1})$, there exists a unique solitary
wave $m(t,x)=\mu(x-ct)$ of the $b$-CH equation \eqref{m} with profile $\mu\in C^{\infty}(\mathbb{R})$ satisfying
$\mu(x)>0$ for $x\in\mathbb{R}$, $\mu'(0)=0$, and $\mu(x)\to k$ as $|x|\to \infty$ exponentially fast. The solitary
wave is orbitally stable in $X_k$ if the mapping
\begin{equation}\label{Q}
k\mapsto Q(\phi):=\int_{\mathbb{R}}\Big(b(\frac{c-k}{c-\phi})-(\frac{c-k}{c-\phi})^{b}-b+1\Big)dx
\end{equation}
is strictly increasing, where $\phi:=k+(1-\partial^2_x)^{-1}(\mu-k)$ is uniquely defined.
\end{lemma}
Thus orbital stability of smooth solitary waves for the $b$-CH equation \eqref{m} can be determined by the
sign of $\frac{dQ(\phi)}{dk}>0$.
Furthermore, the Lemma \ref{le2} implies that the travelling wave solution $u(t,x)=\phi(x-ct)$ of the $b$-CH equation \eqref{CH} is orbitally stable
in $Y_k$ where $Y_k=\{u-k\in H^3(\mathbb{R}): \ u(x)-u''(x)>0, \ x\in\mathbb{R}\}$.

Form \cite{Lafortune},  we obtain the normalized form of the second-order equation \eqref{ode4} after some transformations
\begin{equation}\label{ode2}
-\varphi''+\varphi(1-\varphi)^{b-2}\big(1-\frac{b+1}{2\gamma}\varphi\big)=0, \ \ \varphi\neq1
\end{equation}
where
$$
\zeta=\sqrt{c-k(b+1)}(c-k)^{\frac{b-2}{2}}z, \ \ \psi(z)=k+(c-k)\varphi(\zeta),
$$

$$
z=\int_0^x\frac{1}{(c-\phi(x))^{\frac{b-1}{2}}}dx, \ \ \phi(x)=\psi(z),
$$
and
\begin{equation}\label{r}
\gamma:=\frac{c-k(b+1)}{c-k}.
\end{equation}
Note that \eqref{r}, it is easy to verify that $\gamma\in(0,1)$ owing to $k\in(0,\frac{c}{b+1})$.

On account of $\frac{d\gamma}{dk}=\frac{-bc}{(c-k)^2}<0$ with $b>1, c>0$, the mapping \eqref{Q} is strictly increasing if and only if $\frac{dQ}{d\gamma}<0$, where
\begin{equation}\label{2}
\begin{split}
Q(\phi)&=\int_{\mathbb{R}}\Big(b(\frac{\phi-k}{c-\phi})+1-(\frac{c-k}{c-\phi})^{b}\Big)dx\\
&=\gamma^{-\frac{1}{2}}\int_{\mathbb{R}}\Big(b\varphi(1-\varphi)^{\frac{b-3}{2}}+(1-\varphi)^{\frac{b-1}{2}}
-(1-\varphi)^{-\frac{b+1}{2}}\Big)d\zeta.
\end{split}
\end{equation}

For convenience, denote $x:=\varphi$, $t:=\zeta$, then the equation \eqref{ode2} can be written as
\begin{equation}\label{ex}
-x''+x(1-x)^{b-2}\big(1-\frac{b+1}{2\gamma}x\big)=0,
\end{equation}
and let $x'=y$, we have the planar system as follows
\begin{equation}\label{ex1}
\left\{\begin{aligned}
  &\frac{dx}{dt}=y,\\
  &\frac{dy}{dt}=x(1-x)^{b-2}\big(1-\frac{b+1}{2\gamma}x\big),
     \end{aligned}
  \right.
\end{equation}
with the first integral
\begin{equation}\label{H1}
\bar{H}(x,y)=\frac{(1-x)^{b-1}}{\gamma b(b-1)}\big(2(1-\gamma)+2(1-\gamma)(b-1)x+b(b-1)x^2\big)-y^2=\bar{h}.
\end{equation}

It is clear that there are two singular point $(0,0)$ (the saddle point), $(\frac{2\gamma}{b+1},0)$ (the center), and a singular line $x=1$, where $\frac{2\gamma}{b+1}<1$.

\tikzset{
   flow/.style =
   {decoration = {markings, mark=at position #1 with {\arrow{>}}},
    postaction = {decorate}
   }}

\captionsetup{font={scriptsize}}

\tikzset{global scale/.style={
    scale=#1,
    every node/.append style={scale=#1}
  }
}

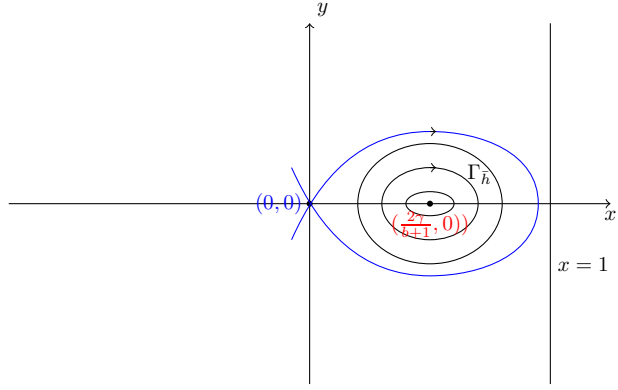
\begin{figure}[htb]
\centering
\begin{tikzpicture}[global scale=0.8]
\fill[black](0,0)circle(0.05);
\draw[->] (-5,0)--(5,0);
\draw(5,0)node[below]{$x$};
\draw[->](0,-3)--(0,3);
\draw(0,3)node[above right]{$y$};
%\draw[flow=0.4] (0,0)ellipse [x radius=2.5, y radius=1.5];
\draw (2,0)ellipse [x radius=0.4, y radius=0.2];
\draw (2,0)ellipse [x radius=0.8, y radius=0.6];
\draw[->](2,0.6)--(2.1,0.6);
\draw (2,0)ellipse [x radius=1.2, y radius=1];
\draw[blue] (-0.3,-0.6) to [out=66,in=180] (2,1.2) to [out=0,in=90] (3.8,0);
\draw[blue] (-0.3,0.6) to [out=-66,in=-180] (2,-1.2) to [out=-0,in=-90] (3.8,0);
\draw[->](2,1.2)--(2.1,1.2);
\draw (0,0)node[blue,left]{$(0,0)$};
\fill[black](2,0)circle(0.05);
\draw (2,0)node[red,below]{$(\frac{2\gamma}{b+1},0)$)};
\draw (4,-3)--(4,3);
\draw (4,-1) node [right]  {$x=1$};
%\draw(-3.16,0)node[below]{$e^-_1$};
%\fill[black](-3,16,0)circle(0.05);
%\draw(3.16,0)node[below right]{$e^+_1$};
%\fill[black](3.16,0)circle(0.05);
%\draw(0,-1.5)node[below left]{$e^-_2$};
%\fill[black](0,-2)circle(0.05);
%\draw(0,1.7)node[right]{$e^+_2$};
%\fill[black](0,2)circle(0.05);
%\draw (-1,1.6)node[red]{$\Gamma$};
\draw(2.5,0.2)node[above right]{$\Gamma_{\bar{h}}$};
\end{tikzpicture}
\caption{Diagram of the homoclinic orbit and the period annulus, where $\Gamma_{\bar{h}}=\{(x,y)|\bar{H}(x,y)=\bar{h},\bar{h}\in(\bar{h}_c,\bar{h}_s)\}$.}
\label{fig1}
\end{figure}

For $b > 1$ and $c > 0$, by the translational invariance, the smooth solitary waves (the homoclinic orbits) profile satisfying $\varphi'(0)=0$ and correspond to the level curve
\begin{equation}\label{orbit}
\frac{(1-x)^{b-1}}{\gamma b(b-1)}\big(2(1-\gamma)+2(1-\gamma)(b-1)x+b(b-1)x^2\big)-y^2=\frac{2(1-\gamma)}{\gamma b(b-1)}.
\end{equation}

Additionally, there exists a punctured neighbourhood of the center $(\frac{2\gamma}{b+1}, 0)$ enclosed by the homoclinic orbit connecting the saddle $(0, 0)$, and the largest such punctured neighborhood is said to be \textit{the period annulus} of $(\frac{2\gamma}{b+1}, 0)$ (see Figue \ref{fig1}).

Next, to better verify the stability criterion \eqref{2}, we consider a new system with the homoclinic orbits \eqref{orbit} as the first integral.

For the homoclinic orbits \eqref{orbit}, it can be written as
\begin{equation}\label{orbit1}
\frac{A(x)}{B(x)}-\frac{b(b-1)y^2}{-B(x)}=\frac{1}{\gamma},
\end{equation}
where
\begin{equation}\label{AB}
\begin{split}
&A(x)=2(1-x)^{b-1}+2(b-1)x(1-x)^{b-1}-2,\\
&B(x)=A(x)+b(b-1)x^2(1-x)^{b-1}.
\end{split}
\end{equation}

In order to separate variables of \eqref{orbit1}, we choose new variables $z,\bar{u}$ as follows:
$$
z=x, \ \  \bar{u}=\sqrt{\frac{b(b-1)}{-B(x)}}y.
$$
Then smooth solitary waves correspond to the level curve takes the form
\begin{equation}\label{H0}
\frac{A(z)}{B(z)}-\bar{u}^2=\frac{1}{\gamma}, \ \ \gamma\in(0,1).
\end{equation}

Note that the level curve \eqref{H0}, let $h=\frac{1}{\gamma}$, it is straightforward to show that
\begin{equation}\label{H}
H(z,\bar{u})=\frac{A(z)}{B(z)}-\bar{u}^2=h, \  h\in(1,\infty),
\end{equation}
is the first integral of the following new planar Hamiltonian system
\begin{equation}\label{ex3}
\left\{\begin{aligned}
  &\frac{dz}{d\tau}=2\bar{u},\\
  &\frac{d\bar{u}}{d\tau}=\frac{2b(b-1)z(1-z)^{b-2}}{B^2(z)}\Big(2-(b+1)z-2(1-z)^b-(b-1)z(1-z)^b\Big),
     \end{aligned}
  \right.
\end{equation}
where $z=1$ is the singular line and $d\tau=\frac{1}{2}\sqrt{\frac{-B(z)}{b(b-1)}}dt$.

The trajectories of smooth solitary waves correspond to the first integral of the system \eqref{ex3}. we denote by $\Gamma_h=\{(z,\bar{u})|H(z,\bar{u})=h, 1<h<\infty\}$ the trajectories of smooth solitary waves (see Figure \ref{fig2}).

\tikzset{
   flow/.style =
   {decoration = {markings, mark=at position #1 with {\arrow{>}}},
    postaction = {decorate}
   }}

\captionsetup{font={scriptsize}}

\tikzset{global scale/.style={
    scale=#1,
    every node/.append style={scale=#1}
  }
}

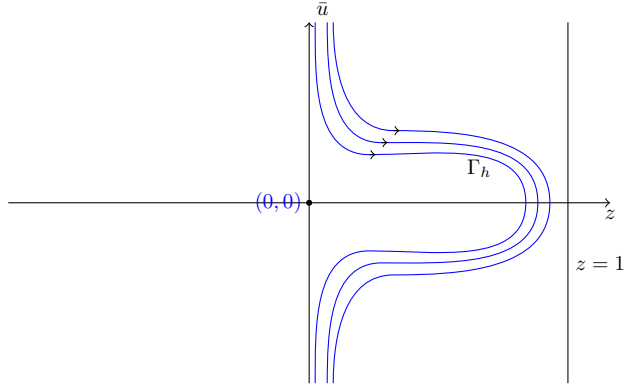
\begin{figure}[htb]
\centering
\begin{tikzpicture}[global scale=0.8]
\fill[black](0,0)circle(0.05);
\draw[->] (-5,0)--(5,0);
\draw(5,0)node[below]{$z$};
\draw[->](0,-3)--(0,3);
\draw(0,3)node[above right]{$\bar{u}$};
%\draw[flow=0.4] (0,0)ellipse [x radius=2.5, y radius=1.5];
%\draw (2,0)ellipse [x radius=0.4, y radius=0.2];
%\draw[->](2,0.6)--(2.1,0.6);
%\draw (2,0)ellipse [x radius=0.8, y radius=0.6];
\draw[blue] (0.1,3) to [out=-90,in=-180] (1,0.8) to [out=0,in=90] (3.6,0);
\draw[blue] (0.1,-3) to [out=90,in=-180] (1,-0.8) to [out=-0,in=-90] (3.6,0);
\draw[->](1,0.8)--(1.1,0.8);
\draw[blue] (0.3,3) to [out=-90,in=-180] (1.2,1) to [out=0,in=90] (3.8,0);
\draw[blue] (0.3,-3) to [out=90,in=-180] (1.2,-1) to [out=-0,in=-90] (3.8,0);
\draw[->](1.2,1)--(1.3,1);
\draw[blue] (0.4,3) to [out=-90,in=-180] (1.4,1.2) to [out=0,in=90] (4,0);
\draw[blue] (0.4,-3) to [out=90,in=-180] (1.4,-1.2) to [out=-0,in=-90] (4,0);
\draw[->](1.4,1.2)--(1.5,1.2);
\draw (0,0)node[blue,left]{$(0,0)$};
%\fill[black](2,0)circle(0.05);
%\draw (2,0)node[red,below]{$(\frac{2\gamma}{b+1},0)$)};
\draw (4.3,-3)--(4.3,3);
\draw (4.3,-1) node [right]  {$z=1$};
%\draw(-3.16,0)node[below]{$e^-_1$};
%\fill[black](-3,16,0)circle(0.05);
%\draw(3.16,0)node[below right]{$e^+_1$};
%\fill[black](3.16,0)circle(0.05);
%\draw(0,-1.5)node[below left]{$e^-_2$};
%\fill[black](0,-2)circle(0.05);
%\draw(0,1.7)node[right]{$e^+_2$};
%\fill[black](0,2)circle(0.05);
%\draw (-1,1.6)node[red]{$\Gamma$};
\draw(2.5,0.3)node[above right]{$\Gamma_h$};
\end{tikzpicture}
\caption{Diagram of the smooth solitary waves correspond to the level curve $\Gamma_{h}$}
\label{fig2}
\end{figure}

Using \eqref{AB}, in view of $z\in(0,1)$ and $b>1$, it is easy to check that
\begin{equation}\label{8}
\begin{split}
& A'(z)=-2b(b-1)z(1-z)^{b-2}<0,\\
& B'(z)=-b(b-1)(b+1)z^2(1-z)^{b-2}<0,\\
&A''(z)=2b(b-1)(1-z)^{b-3}\big((b-1)z-1\big),\\
&B''(z)=b(b-1)(b+2)z(1-z)^{b-3}(bz-2),
\end{split}
\end{equation}
and $A(z)<0$, $B(z)<0$  thanks to $A(0)=0, B(0)=0$.

For the system \eqref{ex3}, denote $f(z)=2-2(1-z)^b-(b+1)z-(b-1)z(1-z)^b$, it can easily be proved that
\begin{equation}\label{f}
f'(z)=(b+1)\big((1-z)^{b-1}+(b-1)z(1-z)^{b-1}-1\big)=\frac{b+1}{2}A(z).
\end{equation}
Since $z\in(0,1)$, we obtain $f'(z)<0$ and $f(z)<0$ due to $f(0)=0$. That is, the system \eqref{ex3} has no singular point with $z\in(0,1)$.

Moreover, by \eqref{2}, it follows that
\begin{equation}
\begin{split}
Q(\phi)=&h^{\frac{1}{2}}\int_{\mathbb{R}}\Big(bz(1-z)^{\frac{b-3}{2}}+(1-z)^{\frac{b-1}{2}}
-(1-z)^{-\frac{b+1}{2}}\Big)dt\\
=&h^{\frac{1}{2}}\int_{\mathbb{R}}\frac{(1-z)^{-\frac{b+1}{2}}}{2}A(z)dt\\
=&h^{\frac{1}{2}}\int_{\Gamma_h}\frac{A(z)(-B)^{\frac{3}{2}}(z)}
{2b^{\frac{1}{2}}(b-1)^{\frac{1}{2}}z(1-z)^{\frac{3b-3}{2}}f(z)}d\bar{u},
\end{split}
\end{equation}
owing to \eqref{ex3}. For simplicity, we denote $A(z)$, $B(z)$ and $f(z)$ as $A$, $B$ and $f$, respectively. That is
\begin{equation}\label{3}
Q(\phi)=h^{\frac{1}{2}}\int_{\Gamma_h}\frac{A(-B)^{\frac{3}{2}}}
{2b^{\frac{1}{2}}(b-1)^{\frac{1}{2}}z(1-z)^{\frac{3b-3}{2}}f}d\bar{u},
\end{equation}
namely $Q$ function.

In view of $\frac{dh}{d\gamma}=-\frac{1}{\gamma^2}<0$, the mapping $Q(\phi)$ \eqref{Q} is strictly increasing if and only if
$\frac{dQ}{dh}>0$.

\section{Preparations of analysis of the stability criterion}\label{sect-3}
To address the monotonicity of the $Q$ function \eqref{3}, we need the following preparations.

As stated in the previous section, we only need to verify that $\frac{dQ}{dh}>0$ for the planar Hamiltonian system \eqref{ex3}.

We find that the $Q$ function is very similar to the \textit{period function} (the period function assigns to each orbit in the period annulus its period) of the center of those planar differential systems for which the first integral $H(x,y)$ has separable variables, i.e., $H(x,y)=F_1(x)+F_2(y)$ (see Figure \ref{fig1}).

In the literatures, much attention is paid to the centers and the monotonicity of period functions of the planar quadratic polynomial systems, see for example \cite{Chicone2,Coppel,Garijo,Gasull,Li1,Zhao} and reference therein.

In \cite{Jordi}, Villadelprat and Zhang considered the monotonicity of the period function of planar Hamiltonian differential systems with the first integral $H(x,y)=F_1(x)+F_2(y)$, where the period function can be written as $$T(\bar{h})=\int_{\Gamma_{\bar{h}}}dt=\int_{\Gamma_{\bar{h}}}\frac{1}{F'_1(y)}dx.$$
Later, the monotonicity of the period function as follows
$$T(\bar{h})=\int_{\Gamma_{\bar{h}}}\frac{g(x)}{l(y)}dx,$$
with the first integral $H(x,y)=F_1(x)+F_2(y)$ was studied by our previous work in \cite{Long}. To solve the convergence problem, we multiply the period function $T(\bar{h})$ by $\bar{h}$ and take the derivative of $\bar{h}T(\bar{h})$ with respect to $\bar{h}$.

In the following proof, we shall adopt the same procedure as in the proof of the monotonicity of the period function of planar Hamiltonian differential systems.

Similarly, for the $Q$ function, denote
$$Q(\phi)=h^{\frac{1}{2}}\int_{\Gamma_h}\frac{A(-B)^{\frac{3}{2}}}
{2b^{\frac{1}{2}}(b-1)^{\frac{1}{2}}z(1-z)^{\frac{3b-3}{2}}f}d\bar{u}\triangleq h^{\frac{1}{2}}\int_{\Gamma_h}g(z)d\bar{u},$$
and we present a consequence to deal with the $Q$ function as follows, which will be applied to prove the main result in the next section.

\begin{lemma}\label{le3}
Assume for $b>1$, the following two hypotheses hold:
\begin{itemize}
  \item [{\rm(H1)}]
      $2(1-z)A'f+(b-1)Af-(1-z)Af'>0$, \  $z\in(0,1)$,
  \item [{\rm(H2)}] $\frac{1}{2}z(1-z)B'+\frac{1}{2}(b-1)zB-(1-z)B<0$, \  $z\in(0,1)$.
\end{itemize}
Then the mapping $Q(\phi)$ \eqref{Q} is strictly increasing with respect to $k$.
\end{lemma}

\begin{proof}
It is sufficient to verify that $\frac{dQ}{dh}>0$ for the planar Hamiltonian system \eqref{ex3}.

For the $Q(\phi)=h^{\frac{1}{2}}\int_{\Gamma_h}g(z)d\bar{u}$,
multiplying both sides of the above equation by $h^{\frac{1}{2}}$, we obtain
\begin{equation}\label{Q1}
h^{\frac{1}{2}}Q(\phi)=h\int_{\Gamma_h}g(z)d\bar{u},
\end{equation}
and
\begin{equation}\label{5}
\begin{split}
h^{\frac{1}{2}}Q(\phi)=&\int_{\Gamma_h}\big(\frac{A}{B}(z)-\bar{u}^2\big)g(z)d\bar{u}\\
=&\int_{\Gamma_h}\big(\frac{A}{B}g\big)(z)d\bar{u}
-\int_{\Gamma_h}g(z)\bar{u}^2d\bar{u}\\
\triangleq&I_1(h)-I_2(h).
\end{split}
\end{equation}
by virtue of the first integral \eqref{H}.

Taking the derivative with respect to $h$ on both sides of the above equality, we have
\begin{equation}\label{6}
\frac{1}{2}h^{-\frac{1}{2}}Q(\phi)+h^{\frac{1}{2}}\frac{dQ(\phi)}{dh}=I'_1(h)-I'_2(h),
\end{equation}
where
\begin{equation}
\begin{split}
I'_1(h)=&\int_{\Gamma_h}\frac{1}{2b^{\frac{1}{2}}(b-1)^{\frac{1}{2}}}\cdot\Big(\frac{A^2}{(1-z)^{b-1}f}\Big)'
\frac{\partial z}{\partial h}\cdot\frac{-(-B)^{\frac{1}{2}}}{z(1-z)^{\frac{1}{2}(b-1)}}d\bar{u}\\
&+\int_{\Gamma_h}\frac{1}{2b^{\frac{1}{2}}(b-1)^{\frac{1}{2}}}\cdot\frac{A^2}{(1-z)^{b-1}f}
\cdot\Big(\frac{-(-B)^{\frac{1}{2}}}{z(1-z)^{\frac{1}{2}(b-1)}}\Big)'\frac{\partial z}{\partial h}d\bar{u}\\
=&\int_{\Gamma_h}\frac{-A(-B)^{\frac{5}{2}}\big(2(1-z)A'f+(b-1)Af-(1-z)Af'\big)}
{4b^{\frac{3}{2}}(b-1)^{\frac{3}{2}}z^2(1-z)^{\frac{5b-5}{2}}f^3}d\bar{u}\\
&+\int_{\Gamma_h}\frac{A^2(-B)^{\frac{3}{2}}\big(z(1-z)B'+\frac{1}{2}(b-1)zB-(1-z)B\big)}
{4b^{\frac{3}{2}}(b-1)^{\frac{3}{2}}z^3(1-z)^{\frac{5b-5}{2}}f^2}d\bar{u},
\end{split}
\end{equation}
and
\begin{equation}
\begin{split}
I'_2(h)=&\Big(\int_{\Gamma_h}\frac{b^{\frac{1}{2}}(b-1)^{\frac{1}{2}}A}{2(1-z)^{\frac{b+1}{2}}\sqrt{-B}}\bar{u}dz\Big)'
=\int_{\Gamma_h}\frac{b^{\frac{1}{2}}(b-1)^{\frac{1}{2}}A}{2(1-z)^{\frac{b+1}{2}}\sqrt{-B}}\cdot\frac{\partial \bar{u}}{\partial h}dz\\
=&\int_{\Gamma_h}-\frac{b^{\frac{1}{2}}(b-1)^{\frac{1}{2}}A}{2(1-z)^{\frac{b+1}{2}}\sqrt{-B}}\cdot\frac{1}{2\bar{u}}dz\\
=&-\int_{\Gamma_h}\frac{A(-B)^{\frac{3}{2}}}{4b^{\frac{1}{2}}(b-1)^{\frac{1}{2}}z(1-z)^{\frac{3b-3}{2}}f}d\bar{u},
\end{split}
\end{equation}
due to \eqref{H} and \eqref{ex3}.

In addition, we must verify that the integral $I'_1(h)$ and $I'_2(h)$ are well defined along the orbit $\Gamma_h$.
It is easy to verify that for $\bar{u}\to0$, the denominator of $I'_1(h)$ and $I'_2(h)$ is not $0$. That is the integral $I'_1(h)$ and $I'_2(h)$ are Riemann integrals, which are well-defined.

And it is simple to know that $\bar{u}\to+\infty$ is singularity of $I'_1(h)$ and $I'_2(h)$.
Using the fact of the first integral \eqref{H}, we have
$$\frac{z}{\frac{A(z)z}{B(z)}-hz}=\frac{1}{\bar{u}^2},$$
and $\bar{u}\to+\infty$,
$$
z=\frac{1}{\bar{u}^2}+O(\frac{1}{\bar{u}^2}),
$$
by means of Lagrange inversion theorem.

For
\begin{equation}\label{I}
-\int_{0}^{+\infty}\frac{-A(-B)^{\frac{5}{2}}\big(2(1-z)A'f+(b-1)Af-(1-z)Af'\big)}
{2b^{\frac{3}{2}}(b-1)^{\frac{3}{2}}z^2(1-z)^{\frac{5b-5}{2}}f^3}d\bar{u},
\end{equation}
the Taylor expansion of the numerator and the denominator of the above function \eqref{I} at the origin have the form
$$
-\int_{0}^{+\infty}\frac{\alpha_1z^{15}+o(z^{16})}{\alpha_2z^{\frac{25}{2}}+o(z^{\frac{27}{2}})}d\bar{u},
$$
where $\alpha_1, \alpha_2$ are parameters related to $b$. It follows that the integral \eqref{I} is absolutely convergent for $\bar{u}\to+\infty$ ($z\to0$), i.e., the integral \eqref{I} is well defined.

Similarly, we can also check that the power of $z$ in the numerator is higher than the denominator of $I'_1(h)$ and $I'_2(h)$ thanks to \eqref{AB} and \eqref{8}, that is integral $I'_1(h)$ and $I'_2(h)$ are well defined.

By \eqref{3}, it is simple to show that
$$\frac{1}{2}h^{-\frac{1}{2}}Q(\phi)=
\int_{\Gamma_h}\frac{A(-B)^{\frac{3}{2}}}{4b^{\frac{1}{2}}(b-1)^{\frac{1}{2}}z(1-z)^{\frac{3b-3}{2}}f}d\bar{u}=-I'_2(h).$$
It follows that
\begin{equation}\label{7}
\begin{split}
h^{\frac{1}{2}}\frac{dQ(\phi)}{dh}=I'_1(h)=&\int_{\Gamma_h}\frac{-A(-B)^{\frac{5}{2}}\big(2(1-z)A'f+(b-1)Af-(1-z)Af'\big)}
{4b^{\frac{3}{2}}(b-1)^{\frac{3}{2}}z^2(1-z)^{\frac{5b-5}{2}}f^3}d\bar{u}\\
&+\int_{\Gamma_h}\frac{A^2(-B)^{\frac{3}{2}}\big(z(1-z)B'+\frac{1}{2}(b-1)zB-(1-z)B\big)}
{4b^{\frac{3}{2}}(b-1)^{\frac{3}{2}}z^3(1-z)^{\frac{5b-5}{2}}f^2}d\bar{u}\\
=&-\int_{0}^{+\infty}\frac{-A(-B)^{\frac{5}{2}}\big(2(1-z)A'f+(b-1)Af-(1-z)Af'\big)}
{2b^{\frac{3}{2}}(b-1)^{\frac{3}{2}}z^2(1-z)^{\frac{5b-5}{2}}f^3}d\bar{u}\\
&-\int_{0}^{+\infty}\frac{A^2(-B)^{\frac{3}{2}}\big(\frac{1}{2}z(1-z)B'+\frac{1}{2}(b-1)zB-(1-z)B\big)}
{2b^{\frac{3}{2}}(b-1)^{\frac{3}{2}}z^3(1-z)^{\frac{5b-5}{2}}f^2}d\bar{u}.
\end{split}
\end{equation}

On account of
$$\frac{-A(-B)^{\frac{5}{2}}}{2b^{\frac{3}{2}}(b-1)^{\frac{3}{2}}z^2(1-z)^{\frac{5b-5}{2}}f^3}<0,$$
and
$$\frac{A^2(-B)^{\frac{3}{2}}}{2b^{\frac{3}{2}}(b-1)^{\frac{3}{2}}z^3(1-z)^{\frac{5b-5}{2}}f^2}>0,$$
and according to $(H1)$ and $(H2)$, we obtain $\frac{Q(\phi)}{dh}>0$. That is the mapping $Q(\phi)$ \eqref{Q} is strictly increasing with respect to $k$. The proof is completed.
\end{proof}

\section{The proof of the main result}\label{sect-4}
In this section, we shall show analytically for any $b>1$ that the stability criterion is satisfied and smooth solitary waves of the the $b$-CH equation \eqref{CH} are orbitally stable.

\begin{proof}

From Lemma \ref{le3}, it is enough to check $(H1)$ and $(H2)$.

$\bullet$ Firstly, we need to verify $(H1)$.

From \eqref{8}, it follows that
\begin{equation}
\begin{split}
&\frac{1}{2}z(1-z)B'+\frac{1}{2}(b-1)zB-(1-z)B\\
=&(1-z)^{b-1}\big((b-1)z^2+(3-b)z-2\big)-(b+1)z+2\\
\triangleq&R(z),
\end{split}
\end{equation}
and it can easily be checked $R(1)=-b+1<0$. The Taylor expansion of the function $R$ at the origin has the form
$$
R(z)=\frac{b}{6}(1-b)(1+b)z^3+o(z^3).
$$
Thus, we have $R'(0)=R''(0)=0$, $R^3(0)=\frac{b}{6}(1-b)(1+b)<0$ and $z\to0+$, $\frac{R(z)}{z^3}<0$.

To determine the sign of $R(z)$, assume that $b=3$, we obtain for $z\in(0,1)$,
$$R(z)=2z^3(z-2)<0.$$
Taking the derivative of $R(z)$ with respect to $z$, we have
\begin{equation}
R'(z)=(b+1)(1-z)^{b-1}\big((b-1)z+1\big)-(b+1).
\end{equation}

In order to prove the theorem, we assertion that if the two equations $R(z)=0$ and $R'(z)=0$ have no the common roots for $z\in(0,1)$, $b>1$, then $R(z)<0$ for $z\in(0,1)$, $b>1$.

If the assertion would not hold, then there exist a parameter $b_1\in(1,+\infty)$ such that $\frac{R(z)}{z^3}>0$. According to the continuous dependence of the solution on the parameters, we can find another parameter $b_0\in(1,+\infty)$ such that $R(z)=0$ and $R'(z)=0$ have one common root for $z\in(0,1)$ in light of $\frac{R(z)}{z^3}<0$ for $b=3$, $z\in(0,1)$. This leads to a contradiction.
It is now obvious that the assertion holds (see Figure \ref{fig3}).

\tikzset{
   flow/.style =
   {decoration = {markings, mark=at position #1 with {\arrow{>}}},
    postaction = {decorate}
   }}

\captionsetup{font={scriptsize}}

\tikzset{global scale/.style={
    scale=#1,
    every node/.append style={scale=#1}
  }
}

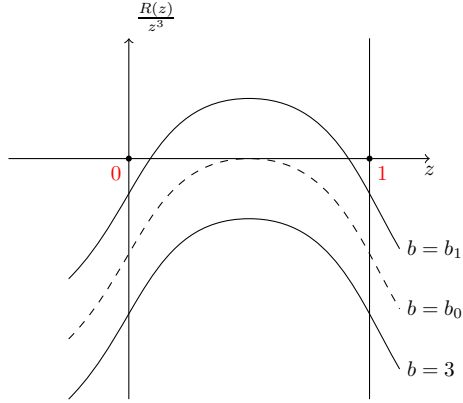
\begin{figure}[htb]
\centering
\begin{tikzpicture}[global scale=0.8]
\fill[black](0,0)circle(0.05);
\draw[->] (-2,0)--(5,0);
\draw(5,0)node[below]{$z$};
\draw[->](0,-4)--(0,2);
\draw(4,-4)--(4,2);
\draw(0,2)node[above right]{$\frac{R(z)}{z^3}$};
\fill[black] (4,0)circle(0.05);
\draw[red](0,0)node[below left]{$0$};
\draw[red](4,0)node[below right]{$1$};
\draw[black] (-1,-4) to [out=45,in=180] (2,-1) to [out=0,in=120] (4.5,-3.5);
\draw(4.5,-3.5)node[right]{$b=3$};
\draw[dashed] (-1,-3) to [out=45,in=180] (2,0) to [out=0,in=120] (4.5,-2.5);
\draw(4.5,-2.5)node[right]{$b=b_0$};
\draw[black] (-1,-2) to [out=45,in=180] (2,1) to [out=0,in=120] (4.5,-1.5);
\draw(4.5,-1.5)node[right]{$b=b_1$};
\end{tikzpicture}
\caption{Diagram of the function $\frac{R(z)}{z^3}$, $z\in(0,1)$, the dotted line corresponds to $R(z)=0$ and $R'(z)=0$ have one common root.}
\label{fig3}
\end{figure}

In other words, we only need to prove that $R(z)=0$, $R'(z)=0$ have no the common roots on $(0,1)$.

Substitute $\nu$ for $(1-z)^{b-1}$ in $R(z)$ and $(1-z)R'(z)$, we can get
$$R(z)=\nu\big((b-1)z^2+(3-b)z-2\big)-(b+1)z+2,$$
and
$$R'(z)=(b+1)\nu\big((b-1)z+1\big)-(b+1),$$
where $R(z)$ and $R'(z)$ are functions of $z, \nu$, and $z\in(0,1)$, $\nu\in(0,1)$.

With the help of elimination by eliminant, the common roots of $R(z)=0$ and $R'(z)=0$ satisfy the following equation
$$b(b+1)(b-1)z^2=0.$$

It is simple to check that for $b>1$ and $z\in(0,1)$, $b(b+1)(b-1)z^2>0$.
It follows that $R(z)=0$ and $R'(z)=0$ have no the common roots for $z\in(0,1)$.

As what we have hoped, that is for $b>1$, $z(1-z)B'+\frac{1}{2}(b-1)zB-(1-z)B<0$, $z\in(0,1)$.

$\bullet$ The next thing to do in the proof is to verify $(H2)$.

We now proceed as in the proof of $(H1)$. If we plug $\nu=(1-z)^{b-1}$ back into $2(1-z)A'f+(b-1)Af-(1-z)Af'$, it can easily be shown that
\begin{equation}
\begin{split}
&2(1-z)A'f+(b-1)Af-(1-z)Af'\\
=&\big(2(b-1)^2z^2-2(b^2-5b+2)z-6b+2\big)\nu^2\\
&+\big(2b(b-1)^2z^2-4(3b-1)z+12b-4\big)\nu+2b(b+1)z-6b+2\\
\triangleq&P(z),
\end{split}
\end{equation}
thanks to \eqref{8}, and it can easily be verified that $P(1)=2(b-1)^2>0$. The Taylor expansion of the function $P$ with respect to $z$ at the origin has the form
$$
P(z)=\frac{1}{6}b^2(b+1)(b-1)^2z^4+o(z^4).
$$
Thus, we have $P'(0)=P''(0)=0$, $P^3(0)=0$, $P^4(0)=\frac{1}{6}b^2(b+1)(b-1)^2>0$ and $z\to0+$, $\frac{P(z)}{z^4}>0$.

Supposed that $b=2$, it is evident that $P(z)=2z^4>0$ for $z\in(0,1)$.
Taking the derivative of $P(z)$ with respect to $z$, we get
\begin{equation}
\begin{split}
(1-z)P'(z)=&(1-z)\big((b-3)A'f+2(1-z)A''f+bAf'\big)\\
=&\big(-4b(b-1)^2z^2+2b(2b^2-9b+5)z+10b^2-6b\big)\nu^2\\
&+\big(-2b(b+1)(b-1)^2z^2+4b^2(b+1)z-12b^2+4b\big)\nu\\
&-2b(b+1)z+2b(b+1)\\
\end{split}
\end{equation}
due to \eqref{8}.

It follows that if for $b>1$ the two equations $P(z)=0$, $P'(z)=0$ have no the common roots, then $P(z)>0$, $z\in(0,1)$.

The common roots of $P(z)=0$ and $P'(z)=0$ satisfy the following equation
$$
z^4\big((b-1)^3z^2+(-12b+4)z+(12b-4)\big)\triangleq z^4l(z)=0.
$$
Now that $b>1$, it can easily be verified that
$$l(z)=(b-1)^3z^2+(12b-4)(1-z)>0, \  z\in(0,1),$$
i.e., $P(z)>0$ for $z\in(0,1)$, $b>1$.

Therefore, we obtain $2(1-z)A'f+(b-1)Af-(1-z)Af'>0$ for $z\in(0,1)$, $b>1$.

Based on the above analyses, the two hypotheses $(H1)$ and $(H2)$ hold.

Hence, by the Lemma \ref{le2}, for any $b>1$, $c>0$ and $k\in(0, \frac{c}{b+1})$, the stability criterion is verified analytically and the smooth solitary waves are orbitally stable. This completes the proof.
\end{proof}

\subsection*{Acknowledgements}
The paper is supported by the National Natural Science Foundation of China (No. 12171491).


\begin{thebibliography}{1}
{\small

\bibitem{Barnes}
L. E. Barnes, A. N. W. Hone, Similarity reductions of peakon equations: the b-family, Theoret. and Math. Phys.
\textbf{212} (2022) 1149-1167.

\bibitem{Camassa}
R. Camassa, D. D. Holm, An integrable shallow water equation with peaked solitons, Phys. Rev. Lett. \textbf{71} (1993) 1661-1664.


\bibitem{Chicone2}
C. Chicone, M. Jacobs, Bifurcation of critical periods for plane vector fields, Trans. Amer. Math. Soc. \textbf{312} (1989) 433-486.


\bibitem{Constantin1}
A. Constantin, J. Escher, Global existence and blow-up for a shallow water equation, Ann. Scuola Norm. Sup. Pisa Cl. Sci.
\textbf{26} (1998) 303-328.


\bibitem{Constantin2}
A. Constantin, J. Escher, Wave breaking for nonlinear nonlocal shallow water equations, Acta Math. \textbf{181} (1998) 229-243.


\bibitem{Constantin3}
A. Constantin, L. Molinet, Orbital stability of solitary waves for a shallow water equation, Phys. D \textbf{157} (2001) 75-89.


\bibitem{Constantin4}
A. Constantin, W. A. Strauss, Stability of peakons, Comm. Pure Appl. Math. \textbf{53} (2000) 603-610.


\bibitem{Constantin5}
A. Constantin, W. A. Strauss, Stability of the Camassa-Holm solitons, J. Nonlinear Sci. \textbf{12} (2002) 415-422.


\bibitem{Coppel}
W. A. Coppel, L. Gavrilov, The period function of a Hamiltonian quadratic system, Differential Integral Equations \textbf{6} (1993) 1357-1365.


\bibitem{Degasperis1}
A. Degasperis, M. Procesi, Asymptotic integrability, in: Symmetry and Perturbation Theory (Rome, 1998), World Sci. Publ., River Edge N.J, 1999, 23-37.


\bibitem{Degasperis2}
A. Degasperis, D. D. Kholm, A. N. I. Khon, A new integrable equation with peakon solutions, Teoret. Mat. Fiz. \textbf{133} (2002) 170-183.


\bibitem{Dullin}
H. R. Dullin, G. A. Gottwald, D. D. Holm, An integrable shallow water equation with linear and nonlinear dispersion, Phys. Rev. Lett. \textbf{87} (2001) 194501.


\bibitem{Garijo}
A. Garijo, J. Villadelprat, Algebraic and analytical tools for the study of the period function, J. Differential Equations \textbf{257} (2014) 2464-2484.


\bibitem{Gasull}
A. Gasull, A. Guillamon, J. Villadelprat, The period function for second-order quadratic ODEs is monotone, Qual. Theory Dyn. Syst. \textbf{4} (2004) 329-352.


\bibitem{Grillakis}
M. Grillakis, J. Shatah, W. Strauss, Stability theory of solitary waves in the presence of symmetry. I, J. Funct. Anal. \textbf{74} (1987) 160-197.


\bibitem{Guo}
B. L. Guo, Z. R. Liu, Periodic cusp wave solutions and single-solitons for the $b$-equation, Chaos Solitons Fractals \textbf{23} (2005) 1451-1463.


\bibitem{Holm1}
D. D. Holm, M. F. Staley, Nonlinear balance and exchange of stability in dynamics of solitons, peakons, ramps/cliffs and leftons in a $1+1$ nonlinear evolutionary PDE, Phys. Lett. A \textbf{308} (2003) 437-444.


\bibitem{Holm2}
D. D. Holm, M. F. Staley, Wave structure and nonlinear balances in a family of evolutionary PDEs, SIAM J. Appl. Dyn. Syst. \textbf{2} (2003) 323-380.


\bibitem{Hone}
A. N. W. Hone,  S. Lafortune, Stability of stationary solutions for nonintegrable peakon equations, Phys. D \textbf{269} (2014) 28-36.


\bibitem{Lafortune}
S. Lafortune, D. E. Pelinovsky, Stability of smooth solitary waves in the $b$-Camassa-Holm equation,
Phys. D \textbf{440} (2022) 133477.


\bibitem{Li}
J. Li, Y. Liu, Q. L. Wu, Spectral stability of smooth solitary waves for the Degasperis-Procesi equation, J. Math. Pures Appl. \textbf{142} (2020) 298-314.


\bibitem{Li1}
J. M. Li, C. Z. Li, C. J. Liu, D. C. Wang, The period function of reversible Lotka-Volterra quadratic centers, J. Differential Equations \textbf{307} (2022) 556-579.


\bibitem{Lin}
Z. W. Lin, Y. Liu, Stability of peakons for the Degasperis-Procesi equation, Comm. Pure Appl. Math. \textbf{62} (2009) 125-146.


\bibitem{Liu}
Y. Liu, Z. Y. Yin, Global existence and blow-up phenomena for the Degasperis-Procesi equation, Comm. Math. Phys. \textbf{267} (2006) 801-820.


\bibitem{Liu1}
Z. R. Liu, J. B. Li, Bifurcations of solitary waves and domain wall waves for KdV-like equation with higher order nonlinearity, Internat. J. Bifur. Chaos Appl. Sci. Engrg. \textbf{12} (2002) 397-407.


\bibitem{Liu2}
Z. R. Liu, R. Q. Wang, Z. J. Jing, Peaked wave solutions of Camassa-Holm equation, Chaos Solitons Fractals \textbf{19} (2004) 77-92.


\bibitem{Long}
T. Long, C. J. Liu, S. Q. Wang, The period function of quadratic generalized Lotka-Volterra systems without complex invariant lines, J. Differential Equations \textbf{314} (2022) 491-517.


\bibitem{Lundmark}
H. Lundmark, J. Szmigielski, Multi-peakon solutions of the Degasperis-Procesi equation, Inverse Problems \textbf{19} (2003) 1241-1245.


\bibitem{Ouyang}
Z. Y. Ouyang, S. Zheng, Z. R. Liu, Orbital stability of peakons with nonvanishing boundary for CH and CH-$\gamma$ equations, Phys. Lett. A \textbf{372} (2008) 7046-7050.


\bibitem{Jordi}
J. Villadelprat, X. Zhang, The Period Function of Hamiltonian Systems with Separable Variables, J. Dynam. Differential Equations \textbf{32} (2020) 741-767.


\bibitem{Whitham}
G. B. Whitham, Linear and nonlinear waves, A Wiley-Interscience Publication. John Wiley $\&$ Sons, Inc., New York, 1999.


\bibitem{Zhao}
Y. L. Zhao, The monotonicity of period function for codimension four quadratic system $Q_4$, J. Differential Equations \textbf{185} (2002) 370-387.

}
\end{thebibliography}
\end{document}